\documentclass[12pt,a4paper]{article}
\usepackage{amssymb,amsfonts,amsmath,amsthm}
\usepackage[utf8]{inputenc}
\usepackage[T1]{fontenc}
\usepackage{indentfirst}
\usepackage[dvips]{graphicx}
\usepackage{dsfont}
\usepackage[dvips]{graphics}
\usepackage[dvips]{graphicx}
\usepackage{dsfont}
\usepackage{geometry}
\usepackage{multido}
\usepackage{pstricks}
\usepackage{pstricks-add}
\usepackage{pst-plot}
\usepackage{accents}

\newtheorem{theorem}{Theorem}[section]
\newtheorem{definition}[theorem]{Definition}
\newtheorem{proposition}[theorem]{Proposition}

\newtheorem{lemma}[theorem]{Lemma}
\newtheorem{corollary}[theorem]{Corollary}
\newtheorem{example}[theorem]{Example}

\newtheorem{remark}[theorem]{Remark}

\def \> {\rightarrow}

\begin{document}

\pagestyle{plain}   %cabe\c{c}alho vazio e p\'{e} de pagina com numera\c{c}\~{a}o ao centro!

\pagenumbering{arabic}  %numera as p\'{a}ginas com algarismos arabicos!

\vskip1mm
 \begin{center}\large{\bf A study of  Blattner-Montgomery duality in weak Hopf algebras of groupoid algebras type}\end{center}

\begin{center}
{\bf Wagner Cortes}\end{center}

\begin{center}{ \footnotesize Instituto de Matem\' atica\\
Universidade Federal do Rio Grande do Sul\\
91509-900, Porto Alegre, RS, Brazil\\
E-mail: {\it wocortes@gmail.com}}\end{center}

\begin{abstract}

We completely describe the nucleous and the image of   the application defined in \cite{N}, when all tensor products are over a field.  Moreover, we study a relationship with the results obtained in  \cite{DFAP}.  
\end{abstract}



\section*{Introduction} 

Weak bialgebras and weak Hopf algebras are generalizations of ordinary bialgebras and Hopf algebras in the following sense: the  axioms are the same, but the multiplicativity of the counit and comultiplicativity of the unit are replaced by weaker axioms. Perhaps the well-known easiest example of a weak Hopf algebra is a groupoid algebra; other examples are face algebras \cite{H}, quantum groupoids \cite{NV} and generalized Kac algebras \cite{Y}. A purely algebraic approach can be found in \cite{BV} and \cite{BNS}.

The  smash products  considered in [8] and [10] are unital algebras since the tensor products are over  suitable algebras. In this case, the application defined in [8] is an isomorphism,  however,  when we consider the same smash products of [8] and [10]  over a field, they are not necessarily unital algebras.   Thus, our purpose is to present a generalization of Blattner-Montgomery  duality obtained in ([8], Theorem 3.3) in the case of groupoid algebras  and the tensor products  over a field. Moreover,  we obtain a relationship between our results and the results obtained in ([5], Section 5). 

It is convenient to point out that our results are different  from [5].  The authors in [5] studied the Blattner-Montgomery's duality in the case of  $B*_{\beta} G\sharp KG^*$, where $G$ is a groupoid, $\beta$ is an action of $G$ in $B$ and $KG^*$  is the dual weak Hopf algebra of $KG$. In this paper, we consider the smash product of type $B\sharp KG\sharp KG^*$, where $KG$ is a weak Hopf algebra acting on $B$ and $KG^*$ is the dual weak Hopf algebra.    In this paper  we complete the studies of the duality of Blattner-Montgomery  of ([5], Section 5)  considering  the algebra  $B\sharp KG\sharp KG^*$ and we obtain a link between our results and the ones in ([5], Section 5).

The paper is organized as follows: in the Section 1 we give some preliminaries about weak Hopf algebras, weak Hopf module algebras, weak smash products and groupoids, see [3] and [4] for more details.  In the Section 2, we present our main results. We consider the application defined in  ([8], Lemma 3.1),  but the tensor products are over a field $K$. In this case, we completely describe the nucleous and the image and   we obtain an exact sequence of algebras where we relate our description of the image with the results   obtained in (\cite{DFAP}, Section 5).

\section{Preliminaries}
In this section, we present some preliminaries that will be necessary  during  this paper.
 
We begin this section with the following definition, see \cite{BP} and \cite{CG} for more details.

\begin{definition} Let $K$ be a field. A  weak $K$-bialgebra $H$ is a $K$-module with a $K$-algebra structure $(\mu,\eta)$ and a $K$- coalgebra structure $(\Delta,\epsilon)$ such that the following conditions are satisfied:

(i) $\Delta(hk) = \Delta(h)\Delta(k)$, for all $h,k\in H$;

(ii)  $\Delta^2(1)	=	1_1\otimes   1_21'_{1} \otimes 1'_2 = 1_1 \otimes 1'_{1}1_2 \otimes 1'_{2}$;

(iii) $\epsilon(hkl)	=	\epsilon(hk_1)\epsilon(k_2l) = \epsilon(hk_2)\epsilon(k_1l)$, for all $h,k,l\in H$.

Moreover, a weak $K$-bialgebra $H$ is a weak Hopf algebra if there exists a $k$-linear map $S:A\rightarrow A$, called the antipode, satisfying the following additional  conditions:

(i)  $x_1S(x_2)=\epsilon(1_1x)1_2$, for all $x\in H$;

(ii) $S(x_1)x_2=1_1\epsilon(x1_2)$, for all $x\in H$; 

(iii)  $S(x_1)x_2S(x_3)=S(x)$, for all $x\in H$.\end{definition}

 We used the Sweedler-Heyneman notation for the comultiplication, namely
$\Delta(h) = h_1\otimes  h_2$ .

%\begin{remark} From [3] and [4], for any weak bialgebra $H$, we have the maps $\epsilon_t:H\rightarrow H$ and $\epsilon_s:H\rightarrow H$ by the formulas

%\begin{center} $\epsilon_t(x)= \epsilon(1_1x)1_2, \epsilon_s(x)= 1_1\epsilon(x1_2)$, \end{center}

%and denote by $H^t$ the image of $\epsilon_t(H)$ and dentote by $H^s$ the image of $\epsilon_s(H)$.\end{remark}

The following definition appears in \cite{BP}. 

\begin{definition}  Let  $G$ be a non-empty set with a binary operation partially defined, i.e,  for $g, h\in  G$ we write  $\exists gh$ when the binary operation $gh$ is defined.  An element $e\in G$ is called identity  if $\exists eg$ and  $\exists ge$, then   $eg = g = ge$ and the set of identities of $G$ will be denoted by  $G_0$. The set $G$ is called groupoid if the  following conditions are satisfied:

(G1)  For all  $g,h,l\in G$, if $\exists g(hl)$ and  $\exists (gh)l$, then $g(hl)=(gh)l$; 

(G2)  For all $g,h,l\in G$,  $\exists g(hl)$ if and only if  $\exists gh$ and  $\exists hl$; 

(G3) For each  $g\in G$ there exists unique elements $s(g),t(g)\in G$ such that  $\exists gt(g)$, $\exists s(g)g$ and  $gt(g) = g = s(g)g$;

(G4) For each $g\in G$ there exists an unique element $g^{-1}\in G$ such that $t(g)=g^{-1}g$ and $s(g)=gg^{-1}$. In this case, $s(g)$ is called source of $g$ and $t(g)$ is called the target of $g$.\end{definition}

The following lemma appears in \cite{BP}.

\begin{lemma}  Let  $G$ be a  groupoid. Then the following statements hold.

(1)  For any $g,h\in G$,  $\exists gh$  if and only if  $t(g) = s(h)$ and in this case  $s(gh) = s(g)$ and $t(gh) = t(h)$. 

(2) For any $g,h\in G$, $\exists (gh)^{-1}$ if and only if  $\exists h^{-1}g^{-1}$ and in this case  $(gh)^{-1} = h^{-1}g^{-1}$.\end{lemma}

 We  set  $G_2 =\{(g,h)\in G\times G:t(g)=s(h)\}=\{(g,h)\in G\times G| \exists gh\}$, i.e., the set of pair of elements of  $G$ where the product exists.

 Let   $G$ be a groupoid and  $K$  a field. Then the  groupoid algebra  $KG$ is the  $K$-vector space  whose  basis is  $\{u_g | g\in G\}$ with usual sum and multiplication rule
$u_gu_h= \left\{\begin{array}{cc} u_{gh}, \,\,if \,\, t(g) = s(h)\\
 	0, \,\, otherwise
 	\end{array}\right\}$, where $u_g$, $g\in G$ is a symbol to represent $g\in G$. 
 	
From \cite{CG},  $KG$  is a weak Hopf algebra with operations  $\Delta : KG\rightarrow  KG \otimes KG$, $\epsilon : KG\rightarrow  K$ and $S:KG\rightarrow KG$  defined  by $\Delta(u_g)=u_g\otimes u_g$, $\epsilon(u_g)=1_K$ and $S(u_g)=u_{g^{-1}}$, for all $g\in G$.
 Moreover,  if $G$ is a groupoid such that $G_0$ is a finite set,  then  $KG$  is an unital  weak Hopf algebra whose identity is  $\sum_{e\in G_0}u_e$, see \cite{CG} for more details.

 Let  $G$ be a  groupoid with a finite number of elements,   $K$  a field and   $KG^{*}$  the dual algebra.  As in \cite{CG}, we have $\Delta : KG^{*}\rightarrow KG^{*}\otimes  KG^{*}$,  $\epsilon : KG^{*}\rightarrow K$ and $S:KG^{*}\rightarrow KG^{*}$ defined by   $\Delta(\rho_g)=\sum_{hl=g}\rho_h\otimes\rho_l= \sum_{t(g)=t(h)} \rho_{gh^{-1}}\otimes \rho_h$, $\epsilon(\sum r_g\rho _g)=\sum_{e\in G_0}r_e$ and $S(a_g\rho_g)=a_g\rho_{g^{-1}}$.   We easily have that $KG^{*}$ is a weak Hopf algebra, see \cite{CG} for more details.
 
 We finish this  section with the definition of weak Hopf module algebra, weak smash product and related subjects, see  \cite{BV}, \cite{BNS} and \cite{N} for more details.

\begin{definition} \label{1} Let $KG$ be the  weak Hopf algebra of groupoid type  and $B$ a $K$-algebra. We say that $B$ is a left  weak  $KG$-module algebra if there exists an action $.:KG\otimes B\rightarrow B$ such that the following conditions are satisfied:

(i) $B$ is a left $KG$-module.

(ii) $\sigma(ab)=(\sigma_1.a)(\sigma_2.b)$

(iii) $\sigma1_B=\epsilon_t(\sigma)1_B.$\end{definition}

Suppose that   $KG$ is a weak Hopf algebra of finite dimension and $B$ a left weak Hopf $KG$-module algebra. Similarly with \cite{N} the weak smash product $B\sharp KG$ is the tensor product $B\otimes_{K} KG$ with usual sum and multiplication rule is \begin{center}$(a\sharp u_{\sigma})(b\sharp u_{\tau})=a(\sigma.b)\sharp u_{\sigma\tau}$.\end{center} Note that, in this case, $1_B\sharp 1_{KG}$ is not necessarily an identity of $B\sharp KG$. Since $KG$ is a finite dimensional weak Hopf algebra, we have that $KG^{*}$ the dual  weak Hopf algebra and we have the double smash product $B\sharp KG\sharp KG^{*}$ as the tensor product $B\otimes_{K} KG\otimes_{K} KG^{*}$ with usual sum and multiplication rule is  $(a\sharp u_{m} \sharp \rho_n)(b\sharp u_s \sharp  \rho_t)=(a\sharp u_m)(\rho_n \rightharpoonup (b\sharp u_s))\sharp \rho_n* \rho_t )$, where the action of $KG^{*}$ on $B\sharp KG$ is defined by \begin{center}$\rho_h \rightharpoonup (b\sharp u_l)=b\sharp \rho_h(u_l)u_l$,\end{center}   and  for each $f,g\in KG^*$,   $f*g$  is the convolution product  defined by $(f*g)(\delta_{\sigma})=f(\delta_{\sigma})g(\delta_{\sigma})$, $\sigma\in G$.

\section{A study  of Blattner-Montgomery's duality}

In this section,  we present the main results of this paper. From now on $K$ will always denote a field, all the tensor products will be over $K$, $H=KG$, where $G$ is a finite groupoid and   the smash products $B\sharp KG$ and $B\sharp KG\sharp KG^{*}$ as   the tensor product $B\otimes_{K} KG$ and \begin{center} $B\otimes_{K} KG\otimes_{K} KG^{*}$,\end{center}  respectively, with usual sum and multiplication rule defined as in the last section. Note that  $1_B\sharp 1_{KG}$ and $1_B\sharp 1_{KG}\sharp 1_{KG^{*}}$ are not necessarily the identities of $B\sharp KG$ and $B\sharp KG\sharp KG^{*}$, respectively. 

The following result was proved in \cite{CG}  and it is used from now on  without further mention.

\begin{proposition} \label{4} Suppose that $B$ is a left weak $KG$-module algebra. Then $\{e.1_B| e\in G_0\}$ is a set of central orthogonal idempotents in $B$ and $B=\oplus_{\{x\in G_0\}}B_x$, where $B_x= B(x.1_B)=B(t(x).1_B)=B_{t(x)}$\end{proposition}

\par  We consider the subsets $A_1=\{g.a\sharp u_g\sharp \rho_h|(g,h)\in G_2\,\, and \,\, s(g)=t(g)\}$, $A_2=\{g.a\sharp u_g\sharp \rho_h|(g,h)\in G_2, \,\, and \,\, s(g)\neq t(g)\}$, $A_3=\{ a\sharp u_g\sharp \rho_h| a\in B, (g,h)\notin G_2\}$, $A_4=\{a\sharp u_g\sharp \rho_h|a\in B_l, a\notin B_g, (l,g)\notin G_2, (g,l)\notin G_2, (g,h)\in G_2\}$,  $A_5=\{a\sharp  u_g\sharp \rho_h|a\in B_l, a\notin B_g, (l,g)\in G_2, (g,l)\notin G_2, (g,h)\in G_2, s(g)\neq t(g)\}$, $A_6=\{a\sharp u_g\sharp \rho_h|a\in B_l, a\notin B_g, (l,g)\notin G_2, (g,l)\in G_2, (g,h)\in G_2\}$, $A_7=\{a\sharp u_g\sharp \rho_h|a\notin  B_g, (g,h)\in G_2, a\in B_l, (l,g)\in G_2, (g,l)\in G_2, s(g)=t(g)\}$, $A_8=\{a\sharp u_g\sharp \rho_h|a\in B_l, a\notin B_g, (l,g)\in G_2, (g,l)\in G_2, (g,h)\in G_2, s(g)\neq t(g)\}$, $A_9=\{a\sharp  u_g\sharp \rho_h|a\notin B_g, a\in B_l, (g,h)\in G_2, (l.g)\in G_2, (g,l)\notin G_2, s(g)\notin t(g)\}$ and $A_{10}=\{ a\sharp u_g\sharp \rho_h| a\notin B_g, a\in B_l, (g,h)\in G_2, (l,g)\in G_2, (g,l)\notin G_2, s(g)=t(g)\}$ of $B\sharp KG\sharp KG^{*}$  and we easily have that $B\sharp KG\sharp KG^{*}=A_1+A_2+A_3+A_4+A_5+A_6+A_7+A_8+A_9+A_{10}$.   

 Analogously to    (\cite{N}, Lemma 3.1)  we have the homomorphim  $\varphi:B\sharp KG \sharp KG^{*}\rightarrow End(B\sharp KG)_{B}$ defined  by $\varphi(b\sharp u_h\sharp \rho_l)(a\sharp u_m)=(b\sharp u_h)(a\sharp \rho_l \rightharpoonup u_m)$ and in 
 the next theorem we completely describe the nucleous and the image.

\begin{theorem}  \label{5} Let $B$ a left weak $KG$-module algebra and $\varphi:B\sharp KG\sharp KG^{*}\rightarrow End_{B}(B\sharp KG)$ defined by \begin{center}$\varphi(a\sharp u_g\sharp \rho_h)(b\sharp u_l)=(a\sharp u_g)(b\sharp \rho_h\rightharpoonup u_l)$.\end{center}   Then $ker \varphi=A_3+A_4+A_5+A_6$ and $A_i\cap ker\varphi=(0)$, for $i\in \{1,2, 7,8,9,10\}$. In this case, $Im\varphi\simeq (B\sharp KG\sharp KG^{*})/ ker\varphi \simeq A_1+A_2+A_7+A_8+A_9+A_{10}$ \end{theorem}

\begin{proof}  Let $a\sharp u_g\sharp \rho_h\in A_3$. Then \begin{center}$\varphi( a\sharp u_g\sharp \rho_h)(b\sharp u_l)=(a\sharp u_g)(b\sharp \rho_h\rightharpoonup u_l)=\left\{\begin{array}{cc}
(a\sharp u_g)(b\sharp u_l)\, \, if\,\, l=h\\
0\,\, if\,\,  l\neq h\\
\end{array}\right\}$.\end{center} Since  $(g,h)\notin G_2$ we have that  $\varphi( a\sharp u_g\sharp \rho_h)(b\sharp u_l)=0$.

 Let $a\sharp u_g\sharp \rho_h\in A_4$. Then \begin{center} $\varphi( a\sharp u_g\sharp \rho_h)(b\sharp u_m)=(a\sharp u_g)(b\sharp \rho_h)\rightharpoonup u_m=\left\{\begin{array}{cc}
(a\sharp u_g)(b\sharp u_m) \,\, if\,\, m=h\\
0 \,\,if\,\, m\neq h\\
\end{array}\right\}$.\end{center}  If $m=h$, then we have that \begin{center}$(a\sharp u_g)(b\sharp u_h)=
a(g.b)\sharp u_{gh}= a(h.1_B)(g.1_B)(g.b)\sharp u_{gh}=a(t(h).1_B)(t(g).1_B)(g.b)\sharp u_{gh}$.\end{center} By the fact that  $(l,g)\notin G_2$, we have that  $a(l.1_B)(h.1_B)(g.1_B)(g.b)\in B_{lg}=0$. Thus, $a\sharp u_g\sharp \rho_h\in ker \varphi$ and it follows that   $A_4\subseteq ker \varphi$.   By similar arguments  we show  that  $A_5\subseteq ker \varphi$. 
Next, let  $a\sharp u_g\sharp \rho_h\in A_6$. Then,  $\varphi( a\sharp u_g\sharp \rho_h)(b\sharp u_m)=(a\sharp u_g)(b\sharp \rho_h)\rightharpoonup u_m)=\left\{\begin{array}{cc}
(a\sharp u_g)(b\sharp u_h)\,\, if\,\, m=h\\
0 \,\,if\,\, m\neq h\\
\end{array}\right\}$. Suppose that $m=h$. Hence, we have that  $\varphi( a\sharp u_g\sharp \rho_h)(b\sharp u_m)=(a\sharp u_g)(b\sharp u_h)=a(g.b)\sharp u_{gh}$ and it follows that    \begin{center}$\varphi( a\sharp u_g\sharp \rho_h)(b\sharp u_m)=a(g.b)\sharp u_{gh}=a(t(l).1_B)(t(g).1_B)(g.b)\sharp u_{gh}=0$,\end{center} because if $t(l)=t(g)$, then we would have $a\in B_g$, which contradicts the fact that $a\notin B_g$.  Consequently, $A_6\subseteq ker\varphi$. 

 We show that   $A_i\cap ker\varphi=(0)$, for $i\in \{1,2, 7,8,9,10\}$.  In fact, let $g.a\sharp u_g\sharp \rho_h\in A_1$ and $s=\sum_{e\in G_0} e.1_B\sharp u_h\in B\sharp KG\sharp KG^*$. Then, \begin{center}$\varphi(g.a\sharp u_g\sharp \rho_h)(\sum_{e\in G_0} e.1_B\sharp u_h)=\sum_{e\in G_0}(g.a\sharp u_g)(e.1_B\sharp u_h)=(g.a\sharp u_h)$.\end{center}  and we have that  $A_1\nsubseteq ker\varphi$. Thus,  $A_1\cap \ker\varphi=(0)$. Let $a\sharp u_g\sharp \rho_h \in A_8$ and $(g^{-1}l).1_B\sharp 1_{KG}\in B\sharp KG$.  Then $\varphi(a\sharp u_g\sharp \rho_h)((g^{-1}l).1_B\sharp 1_{KG})=(a\sharp u_g)((g^{-1}l).1_B\sharp u_h)=a(g(g^{-1}l).1_B)\sharp u_{gh}=a(s(g)(l.1_B))\sharp u_{gh}=a(l.1_B)\sharp u_{gh}=a\sharp u_{gh}$ and we obtain that   $A_8\nsubseteq ker\varphi$ and  $A_8\cap ker\varphi=(0)$.  Let $a\sharp u_g \sharp \rho_h\in A_{10}$ and $t(l).1_B\sharp u_h\in B\sharp KG$. Then, $\varphi(a\sharp u_g\sharp \rho_h)(t(l)1_B\sharp u_h)=a\sharp u_{gh}$. Hence, $A_{10}\nsubseteq ker\varphi$ and we obtain that $A_{10}\cap ker\varphi =(0)$. The  same arguments used to prove that $A_1\cap ker\varphi=0$, $A_8\cap ker\varphi =0$ and $A_{10}\cap ker\varphi=0$ can be used to prove that   $A_2\cap ker\varphi=0$, $A_7\cap ker\varphi=0$ and $A_9\cap ker\varphi=0$, respectively. 
 
 Now, by the fact that  $B\sharp KG\sharp KG^{*}=A_1+A_2+A_3+A_4+A_5+A_6+A_7+A_8+A_9+A_{10}$  we have that $ker \varphi=A_3+ A_4 + A_5+A_6$.
 
 So, $Im\varphi\simeq (B\sharp KG\sharp KG^{*})/ ker\varphi \simeq A_1+A_2+A_7+A_8+A_9+A_{10}$\end{proof}   
 
 During the rest of this paper, we use the sets $A_1$, $A_2$, $A_3$, $A_4$, $A_5$, $A_6$, $A_7$, $A_8$, $A_9$, $A_{10}$ and the decomposition $B\sharp KG\sharp KG^*=A_1+A_2+A_3+A_4+A_5+A_6+A_7+A_8+A_9+A_{10}$ mentioned as before.

During  the next three results  we proceed to study with more details about the algebra $Im\varphi$

\begin{proposition} \label{7} $A_1+A_7+A_{10}$ is a subalgebra of $B\sharp KG\sharp KG^{*}$.\end{proposition}

\begin{proof}  Let $f_1=g.a\sharp u_g\sharp \rho_h\in A_1$ and $f_2=b\sharp u_{g'}\sharp \rho_{h'}\in A_7$. Then \begin{center} $f_1f_2=(g.a\sharp u_g)(b\sharp \delta_{hk^{-1},g'}u_{g'}\sharp \delta_{k,h'}\rho_k).$ (*) \end{center} 
Note that $hk^{-1}=g'\Leftrightarrow g'^{-1}h=k$. If either $s(g')\neq t(g)=s(h)$ or  $s(g')=t(g)=s(h)$ and $t(h)\neq t(h')$, then we have that $f_1 f_2=0$. Now, suppose that $s(g')=t(g)=s(h)$ and $t(h)=t(h')$. Thus we can take $k=g'^{-1}h=h'$. Hence $f_1f_2=(g.a\sharp u_g)(b\sharp u_{g'})\sharp \rho_h=g(ab)\sharp u_{gg'}\sharp \rho_h$. By the fact that $s(g)=t(g)$, we have that $s(gg')=t(gg')$ and it follows that  $(g,gg')\in G_2$ and $(gg',g)\in G_2$. Consequently, either $f_1f_2\in A_1$ or $f_1f_2\in A_7$. So, $A_1A_7\subseteq A_1+A_7+A_{10}$. 

Using similar methods we show that  $A_7A_{10}\subseteq A_1+A_7+A_{10}$, $A_7A_1\subseteq A_1+A_7+A_{10}$,  $A_{10}A_{7}\subseteq A_1+A_7+A_{10}$, $A_1A_{10}\subseteq A_1+A_7+A_{10}$ and $A_{10}A_1\subseteq A_1+A_7A_{10}$.\end{proof}

\begin{proposition} \label{8} The element $y=\sum_{l\in G}l.1_B\sharp u_{t(l)}\sharp \sum_{t(l)=s(g)}\rho_g$ is the identity of $A_1+A_7+A_{10}$.\end{proposition}

\begin{proof} It is not difficult  to see that $y\in A_1+A_7+A_{10}$. Let $a\sharp u_g\sharp \rho_h\in A_7$. Then \begin{center} $(a\sharp u_g\sharp \rho_h)(\sum_{l\in G}l1_B\sharp u_{t(l)}\sharp \sum_{t(l)=s(n)}\rho_{n})=(a\sharp u_g)(\sum_{l\in G}l.1_B\sharp \delta_{hk^{-1},t(l)}u_{t(l)})\sharp \delta_{ k,n}\rho_k.$ (*)\end{center}

Note that there exists $l\in G$ such that  $t(l)=s(h)$ and in this case we may assume that $l=t(g)$. Thus, we obtain that $k=t(l)h=h$  and we have  that there exists  $n\in G$ with $t(n)=s(g)$ such that $n=h$. Hence, $(*)=(a\sharp u_g)(t(g).1_B\sharp u_{t(g))}\sharp  \rho_h=a(t(g)1_B)\sharp u_g\sharp \rho_h=a\sharp u_g\sharp \rho_h$. 

 We also consider the following equality  \begin{center} $(\sum_{l\in G} l.1_B\sharp u_{t(l)}\sharp \sum_{t(l)=s(n)}\rho_n)(a\sharp u_g\sharp \rho_h)=(\sum_{l\in G} (l.1_B\sharp u_{t(l)})(a\sharp \delta_{nk^{-1},g}u_g)\sharp \delta_{k,h}\rho_k$. \end{center} Note that we can take $n=gh$ and $k=h$. Thus, the last equality becomes $(\sum_{l\in G}l.1_B\sharp u_{t(l)})(a\sharp u_g\sharp \rho_h)=\sum_{l\in G}t(l).a\sharp u_{t(l)g}\sharp \rho_h)$. Since we can take $l=s(g)$, then $(\sum_{l\in G}l.1_B\sharp u_{t(l)})(a\sharp u_g\sharp \rho_h)=(s(g).a\sharp u_g\sharp \rho_h)$. By the fact that $a\sharp u_g\sharp \rho_h\in A_7$, we have that $a\in B_m$, for some $m\in G$ such that $(m,g)\in G_2$ and it follows that $t(m)=s(g)$. Hence, $(s(g).a\sharp u_g\sharp \rho_h)=a\sharp u_g\sharp \rho_h$. 

By similar arguments we  show that \begin{center} $(\sum_{l\in G}l1_B\sharp u_{t(l)}\sharp \sum_{t(l)=s(n)}\rho_n)(a\sharp u_g\sharp \rho_h)=a\sharp u_g\sharp \rho_h$, for all $a\sharp u_g\sharp \rho_h\in A_{10}$.\end{center}

Finally, let $g.a\sharp u_g\sharp \rho_h\in A_1$. Then \begin{center} $(\sum_{l\in G}l1_B\sharp u_{t(l)}\sharp \sum_{t(l)=s(n)}\rho_{n})(g.a\sharp u_g\sharp \rho_h)=\sum_{l\in G}(l.1_B\sharp u_{t(l)})(g.a\sharp \delta_{nk^{-1},g}u_g)\sharp \delta_{k,h}\rho_k).$(* *)\end{center}

Note that there exists $l\in G$ such that $t(l)=s(g)$. In this case, we have that  $k=h$ and $n=gh$. Consequently, $(**)=(l.1_B\sharp u_{t(l)})(g.a\sharp u_g)\sharp \rho_h=g.a\sharp u_g\sharp \rho _h$. By similar arguments  as before we have that  \begin{center}$(g.a\sharp u_g\sharp \rho_h)(\sum_{l\in G}l1_B\sharp u_{t(l)}\sharp \sum_{t(l)=s(n)}\rho_{n})=g.a\sharp u_g\sharp \rho_h$. \end{center} \end{proof}

One may ask  if the element $y=\sum_{l\in G}l.1_B\sharp u_{t(l)}\sharp \sum_{t(l)=s(g)}\rho_g$  of the Proposition 2.4  is the identity of $Im\varphi$, but the next proposition answers this question in the negative.

\begin{proposition} \label{9} The element $y=\sum_{l\in G}l.1_B\sharp u_{t(l)}\sharp \sum_{t(l)=s(g)}\rho_g$   is in the left annihilator of $A_2$, the right annihilator of $A_8$ and the right annihilator of $A_9$.\end{proposition}

\begin{proof} Let $g.a\sharp u_g\sharp \rho_h\in A_2$. Then \begin{center}$y(g.a\sharp u_g\sharp \rho_h)=\sum_{l\in G} ((l.1_B)\sharp u_{t(l)})(g.a\sharp \delta_{gk^{-1},g}u_g)\sharp \delta_{k,h}\rho_k$. (1)\end{center}  By the fact that $(g,h)\notin G_2$, then $(1)=0$.
Next, let $a\sharp u_g\sharp \rho_h$. Then \begin{center}$(a\sharp u_g\sharp \rho_h)y=(a\sharp u_g)(\sum_{l\in g}l.1_B\sharp \delta_{hk^{-1},t(l)}u_{t(l)})\sharp \delta_{k,n}\rho_k.$ (2) \end{center} Following the same arguments presented  in Theorem 2.2 we have that $(2)= s(g).a\sharp u_g\sharp \rho_h$. Since $a\in B_l$, for some $l\in G$ and $s(g)\neq t(g)$, then $s(g).a\sharp u_g\sharp \rho_h=0$. By similar arguments we show that $A_9y=0$. \end{proof}

Now, we are ready to prove the second main result of the paper.

\begin{theorem} \label{10} There exists a subalgebra $S$ of $A\sharp KG \sharp KG^{*}$ that contains an unital subalgebra $S'$ such that $S=S'\oplus T$ as $K$-algebras  and an ideal $D$ of $A\sharp KG\sharp KG^{*}$ such that $A\sharp KG\sharp KG^{*}=D\oplus  S$.\end{theorem}

\begin{proof} By  the  Propositions 2.3 and 2.4, $S=Im\varphi$ contains an unital subalgebra $S'=A_1+ A_7+ A_{10}$. Now, applying  the Propositions 2.4 and 2.5  we obtain that $S=(A_1+A_7+A_{10})\oplus (A_2+A_8+A_9)$ as $K$-algebras. So, by Theorem 2.2 we have that    $A\sharp KG\sharp KG^*=S\oplus D$, where $D=Ker\varphi$.\end{proof}

\begin{remark} \label{11} By Theorem \ref{10}, the map $\varphi:A\sharp KG\sharp KG^{*}\rightarrow End(A\sharp KG)$ given by $\varphi(a\sharp u_g\sharp \rho_h)(b\sharp u_l)=(a\sharp u_g)(b\sharp \rho_h\rightharpoonup u_l)$ induces  the following exact sequence of $K$-algebras
\begin{center} $0\rightarrow D\rightarrow  A\sharp KG\sharp KG^{*}\rightarrow \varphi(S)\rightarrow 0.$\end{center}\end{remark}

\par Next, we present an example where we apply our results.

\begin{example} \label{3} Let $G = \{t(g),s(g),g,g^{-1}\}$ be a groupoid  with $G_0=\{t(g), s(g)\}$ such that $s(g)\neq t(g)$ and $B = Ke_1\oplus 
Ke_2\oplus  Ke_3$, where $K$ is a field and $e_1$, $e_2$, $e_3$ are pairwise orthogonal 
central idempotents with sum $1_B$. We define the action of $KG$ as follows: $g.(a_1e_1+a_2e_2+a_3e_3)=a_1e_2+a_2e_1+a_3e_1$, $g^{-1}(a_1e_1+a_2e_2+a_3e_3)=a_1e_2+a_2e_1+a_3e_2$, $t(g)|_{Ke_1\oplus Ke_2}=id_{Ke_1\oplus Ke_2}=s(g)|_{Ke_1\oplus Ke_2}$, $t(g)(ae_3)=ae_2$ and $s(g)(ae_3)=ae_1$.  

%Note that, by (\cite{DFAP},  Lemma 5.3)  we have that $(KG)_t=Kt(g)\oplus Ks(g)$ and $((KG)^*)_{t}=Kt(g)\oplus Ks(g)\oplus Kg\oplus Kg^{-1}$. Let $u_g\sharp \rho_g\in KG\sharp KG^*$. Then $u_g\sharp \rho_g=u_g\otimes\rho_g=u_g\leftharpoonup \rho_h \otimes 1_{KG^*}=\delta_{g,g}u_g\otimes 1_{KG^{*}}=u_g\otimes 1_{KG^{*}}$. Hence, in $B\sharp KG\sharp KG^*$ the elements of the form $a\sharp u_g\sharp \rho_g$  are not zero.  

By  Theorem 2.2 we have that $A_1=A_{s(g)}\sharp u_{s(g)}\sharp \rho_{s(g)}+A_{s(g)}\sharp u_{t(g)}\sharp \rho_g+A_{t(g)}\sharp u_{t(g)}\sharp \rho_{t(g)}+A_{t(g)}\sharp u_{t(g)}\sharp \rho_{g^{-1}}$, $A_2=A_{s(g)}\sharp u_{g^{-1}}\sharp \rho_{g}+A_{s(g)}\sharp u_{g^{-1}}\sharp \rho_{s(g)}+A_{t(g)}\sharp u_g\sharp\rho_{t(g)}+A_{t(g)}\sharp u_g\sharp \rho_{g^{-1}}$, $A_3=A_{s(g)}\sharp u_{s(g)}\sharp \rho_{t(g)}+A_{s(g)}\sharp u_{s(g)}\sharp \rho_{g^{-1}}+A_{s(g)}\sharp u_{t(g)}\sharp \rho_{s(g)}+A_{s(g)}\sharp u_{t(g)}\sharp \rho_g+A_{s(g)}\sharp u_g\sharp \rho_{s(g)}+A_{s(g)}\sharp u_g\sharp \rho_g+A_{s(g)}\sharp u_{g}\sharp \rho_{t(g)}+A_{s(g)}\sharp u_{g^{-1}}\sharp \rho_{g^{-1}}+A_{t(g)}\sharp u_{s(g)}\sharp \rho_{t(g)}+A_{t(g)}\sharp u_{s(g)}\sharp \rho_{g^{-1}}+A_{t(g)}\sharp u_{t(g)}\sharp \rho_{s(g)}+A_{t(g)}\sharp u_{t(g)}\sharp \rho_{g}+A_{t(g)}\sharp u_{g^{-1}}\sharp \rho_{g^{-1}}+A_{t(g)}\sharp u_{g^{-1}}\sharp \rho_{t(g)}+A_{s(g)}\sharp u_{t(g)}\sharp \rho_{s(g)}$, $A_4= A_{s(g)}\sharp u_{t(g)}\sharp \rho_{t(g)}+A_{s(g)}\sharp u_{t(g)}\sharp \rho_{g^{-1}}+A_{t(g)}\sharp u_{s(g)}\sharp \rho_{s(g)}+A_{t(g)}\sharp u_{s(g)}\sharp \rho_{s(g)}$, $A_5=A_{s(g)}\sharp u_g\sharp \rho_{t(g)}+A_{s(g)}\sharp u_g\sharp \rho_{g^{-1}}+A_{t(g)}\sharp u_{g^{-1}}\sharp \rho_{s(g)}+A_{t(g)}\sharp u_{g^{-1}}\sharp \rho_{g}$.

Thus, $Ker\varphi=A_3+A_4+A_5$, $Im\varphi=\varphi(A_1+A_2)$,  $y=g1_A\sharp u_{t(g)}\rho_{t(g)}+g1_B\sharp u_{t(g)}\sharp \rho_{g^{-1}}+g^{-1}1_B\sharp u_{s(g)}\sharp \rho_{s(g)}+g^{-1}1_B\sharp u_{s(g)}\sharp \rho_g$ is the identity of element of $A_1$   and $yA_2=0$. Moreover, we have the following exact sequence \begin{center}  $0\rightarrow Ker\varphi\rightarrow A\sharp KG\sharp KG^{*}\rightarrow Im\varphi\rightarrow 0$\end{center}
\end{example}

Accordingly (\cite{DFAP}, Section 2)  an action of a groupoid $G$ on a $K$-algebra $T$ is a pair $\beta= (\{E_g\}_{g\in G},\{\beta_g\}_{g\in G})$, where for each $g\in G$, $E_g=E_{gg^{-1}}$ is an ideal of $T$ and $\beta_g:E_{g^{-1}} \rightarrow E_g$ is an isomorphism of rings satisfying the following conditions:

(i)  $\beta_e$ is the identity map $I_{E_e}$ of $E_e$, for all $e\in G_0$,

(ii) $\beta_g \beta_h(r)=\beta_{gh}(r)$, for all $(g,h)\in G_2$ and $r\in E_{h^{-1}}=E_{(gh)^{-1}}$.

Let $\beta=(\{E_g\}_{g\in G},\{\beta_g\}_{g\in G})$ be an action of $G$ on a $K$-algebra $T$.  Accordingly (\cite{DFAP}, Section 2), the skew groupoid ring $T*_{\beta} G$ corresponding to $\beta$ is defined
as  $\oplus_{g\in G} E_g\delta_g$,
in which the $\delta_g$’s are symbols, with the usual sum  and multiplication rule  \begin{center} $(x\delta_g)(y\delta_h)=\left\{\begin{array}{cc} x\beta_g(y)\delta_{gh}, \,\,if \,\, (g,h)\in G_2\\
 	0, \,\, otherwise
 	\end{array}\right\}$\end{center} 

Now, we are ready  to study    a relationship between  our description and  the one obtained in   (\cite{DFAP}, Section 5).  

\begin{theorem}  \label{12}  Let $G$ be a finite groupoid, $\beta=(\{E_g\}_{g\in G},\{\beta_g\}_{g\in G})$ an action of $G$ on $B$. Suppose that $B$ is a weak $KG$-module algebra  and consider the same assumptions of Theorem 2.2.  Then  we have the  exact sequence    \begin{center} $0\rightarrow D_1 \rightarrow B*_{\beta}G\sharp KG^{*}\rightarrow \varphi(\Psi(C))\rightarrow 0$\end{center}  where $D_1=\oplus_{(g,h)\notin G_2} E_g\delta_g\sharp \rho_h$, $C=\oplus_{(g,h)\in G_2}E_g\delta_h\sharp \rho_h$ and $\Psi:B*_{\beta} G\sharp KG^*\rightarrow B\sharp KG\sharp KG^*$ is defined by $\Psi(a_g\delta_g\sharp \rho_h)=a_g\sharp u_g\sharp \rho_h$. \end{theorem} 
   
   \begin{proof}  First, we easily have that $D_1=ker\varphi\circ \psi$. By (\cite{DFAP}, Proposition 5.1), \begin{center} $B*_{\beta}G\sharp KG^*=C\oplus D_1$\end{center}  and we have that   $(\varphi\circ \Psi)(A*_{\alpha}G\sharp KG^{*})=\varphi(\Psi (C\oplus D_1))=\varphi(\Psi (C))$. It is not difficult to see that $\Psi(C)\subseteq A_1+A_2+A_3$.

We claim that $\Psi(B_0)=A_1$, where $B_0=\oplus_{(g,h)\in G_2, s(g)=t(g)} E_g\delta_g\sharp \rho_h$. In fact, for each element $g.a\sharp u_g\sharp \rho_h\in A_1$ we have that $\Psi(g.a \delta_g\sharp \rho_h)=g.a\sharp u_g\sharp \rho_h$ which implies that $A_1\subseteq \Psi(B_0)$. The other inclusion is trivial. Hence, $\Psi(B_0)=A_1$ is an algebra with identity. 

\end{proof}

Using Theorem \ref{12} and (\cite{DFAP}, Theorem 3.7) we have the following result.

 \begin{corollary}  With  the same assumptions of Theorem \ref{12} we have that $$A_1=\{g.a\sharp u_g\sharp \rho_h|(g,h)\in G_2\,\, and \,\, s(g)=t(g)\}$$ is an homomorphic image of an matrix algebra.
 
 \end{corollary}

%\par   We consider the sets  $\overline{A_1}=\{g.a\sharp u_g\sharp 1_{KG^*}|s(g)=t(g)\}$, $\overline{A_{10}}=\{a\sharp u_g\sharp 1_{KG^{*}}| a\in A_l, s(g)=t(g), a\notin A_g, (g,l)\in G_2\,\, and\,\, (l,g)\in G_2\}$ and $\overline{A_7}=\{a\sharp u_g\sharp 1_{KG^*}|a\in A_l, s(g)=t(g), a\notin A_g, (l,g)\in G_2, (g,l)\notin G_2\}$. 

%We define $\Psi:B\sharp KG\rightarrow B\sharp KG\sharp KG^*$ by $\Psi(a\sharp u_g)=a\sharp u_g\sharp 1_{KG^*}$. It is not difficult to see that $\Psi$ is a monomorphism of algebras.

%The proof of the following lemma  follows the same ideas of Proposition \ref{7}.

%\begin{lemma} \label{13}$\overline{A_1}+\overline{A_{10}}+\overline{A_7}$ is a subalgebra of $B\sharp KG\sharp KG^{*}$ contained in $A_1+A_{10}+A_7$.\end{lemma}

%\par Let $R\subseteq S$ be a extension of rings.  $S$ is said to be separable over $R$ if there exists $e=\sum_{i=1}^{n} s_i\otimes t_i\in S\otimes_{R} S$ such that $e^2=e$ and  for any $r\in R$ we have that $\sum_{i=1}^{n} rs_i\otimes t_i=\sum_{i=1}^{n} s_i\otimes t_ir$ and $\sum_{i=1}^{n}s_it_i=1_S$.

%We finish this article with the follwing result.

%\begin{proposition} \label{14} $A_1+A_{10}+A_7$ is a separable algebra over  $\overline{A_1}+\overline{A_{10}}+\overline{A_7}$.\end{proposition}

%\begin{proof} Let $e=\sum_{l\in G}g.1_A\sharp u_{t(g)}\sharp \sum_{t(g)=s(h)}\rho_h$. Then $e$ is an idempotent and we easily have that $e\otimes e$ is the required element for \begin{center}$(A_1+A_{10}+A_7)\otimes_{\overline{A_1}+\overline{A_{10}}+\overline{A_7}} (A_1+A_{10}+A_7)$.
.%\end{center} \end{proof}  

\end{document}